\documentclass[10pt, reqno]{amsart}
\usepackage{graphicx, amssymb, amsmath, amsthm, color, slashed, cite}
\numberwithin{equation}{section}

\usepackage{hyperref}

\let\Re=\undefined\DeclareMathOperator*{\Re}{Re}
\let\Im=\undefined\DeclareMathOperator*{\Im}{Im}

\newcommand{\R}{\mathbb{R}}
\newcommand{\C}{\mathbb{C}}

\newcommand{\eps}{\varepsilon}

\newtheorem{theorem}{Theorem}[section]

\newtheorem{lemma}[theorem]{Lemma}

\newtheorem{proposition}[theorem]{Proposition}

\theoremstyle{definition}

\theoremstyle{remark}

\newcommand{\qtq}[1]{\quad\text{#1}\quad}

\begin{document}

\title[Inhomogeneous NLS]{Scattering below the ground state for the intercritical non-radial inhomogeneous NLS}

\author[Cardoso]{Mykael Cardoso}
\address{Department of Mathematics, UFPI, Brazil}
\email{mykael@ufpi.edu.br}
\author[Farah]{Luiz Gustavo Farah}
\address{Department of Mathematics, UFMG, Brazil}
\email{farah@mat.ufmg.br}
\author[Guzm\'an]{Carlos M. Guzm\'an}
\address{Department of Mathematics, UFF, Brazil}
\email{carlos.guz.j@gmail.com}
\author[Murphy]{Jason Murphy} 
\address{Department of Mathematics \& Statistics, Missouri S\&T, USA}
\email{jason.murphy@mst.edu}

\begin{abstract} We consider the focusing inhomogeneous nonlinear Schr\"odinger equation
\[
i\partial_t u + \Delta u + |x|^{-b}|u|^\alpha u = 0\qtq{on}\R\times\R^N,
\]
with $N\geq 2$, $0<b<\min\{\tfrac{N}{2},2\}$, and $\tfrac{4-2b}{N}<\alpha<\tfrac{4-2b}{N-2}$. These constraints make the equation mass-supercritical and energy-subcritical.  We extend the results of Farah--Guzm\'an \cite{FG2} and Miao--Murphy--Zheng \cite{MMZ} and prove scattering below the ground state with general initial data. 
\end{abstract}

\maketitle

\section{Introduction}

We consider the focusing inhomogeneous nonlinear Schr\"odinger equation
\begin{equation}\label{INLS}
\begin{cases}
& i\partial_t u + \Delta u + |x|^{-b}|u|^\alpha u = 0, \\
& u|_{t=0}=u_0 \in H^1(\R^N),
\end{cases}
\end{equation}
where $u:\R\times\R^N\to\C$ is a complex-valued function of space-time.  Here we work in dimensions $N\geq 2$ and choose the parameter $b$ so that $0<b<\min\{\tfrac{N}{2},2\}$.  The power $\alpha$ is chosen so that $\tfrac{4-2b}{N}<\alpha<\tfrac{4-2b}{N-2}$, where here and below we understand the upper bound to be $\infty$ when $N=2$. These constraints guarantee that \eqref{INLS} is \emph{mass-supercritical} but \emph{energy-subcritical}, as we now explain.

The equation \eqref{INLS} enjoys the scaling symmetry $u(t,x)\mapsto \lambda^{\frac{2-b}{\alpha}}u(\lambda^2 t,\lambda x)$, which identifies the unique invariant homogeneous $L^2$-based Sobolev space of initial data as $\dot H^{s_c}$, where 
\[
s_c:=\tfrac{N}{2}-\tfrac{2-b}{\alpha}.
\]
When $s_c=0$, the critical space is $L^2$, which is naturally associated to the conserved \emph{mass} of solutions, defined by
\[
M[u] = \int |u|^2 \,dx. 
\]
On the other hand, when $s_c=1$, the critical space is $\dot H^1$, which is naturally associated to the conserved \emph{energy} of solutions, defined by 
\[
E[u] = \int \tfrac12 |\nabla u|^2 - \tfrac1{\alpha+2}|x|^{-b}|u|^{\alpha+2}\,dx. 
\]
The constraints on $\alpha$ guarantee that $0<s_c<1$, which we call \emph{intercritical} (i.e. mass-supercritical and energy-subcritical).

The inhomogeneous NLS model has been the subject of a great deal of recent mathematical interest.  The well-posedness problem has been studied in works such as \cite{Farah, GENSTU, CARLOS, Dinh4, LeeSeo, CFG20}, with a focus on treating as wide of a range of parameters as possible (including taking $b$ to be as large as possible).  We will record the precise well-posedness result that we need in Section~\ref{S:notation} below. 

Our interest in this work is the problem of scattering for \eqref{INLS}.  Here we say that a solution to \eqref{INLS} \emph{scatters} (in $H^1$) if there exist $u_\pm\in H^1$ such that 
\[
\lim_{t\to\pm\infty}\|u(t) - e^{it\Delta}u_\pm\|_{H^1(\R^N)}=0,
\]
where $e^{it\Delta}$ is the free Schr\"odinger propagator, defined as the Fourier multiplier operator with symbol $e^{-it|\xi|^2}$.  Standard local well-posedness arguments show that scattering holds for solutions with sufficiently small initial data. On the other hand, if we let $Q$ denote the ground state solution to
\[
\Delta Q - Q + |x|^{-b}Q^{\alpha+1}=0
\]
(see \cite{GENSTU, Yanagida}), then we find that $u(t,x)=e^{it}Q(x)$ is a global, non-scattering solution to \eqref{INLS}.  Our goal is to find the sharp scattering threshold for solutions that are `below the ground state' in an appropriate sense.  In particular, we will describe the threshold in terms of scale-invariant quantities related to the mass and energy of the initial data. 

Our main result is the following.

\begin{theorem}\label{T} Suppose $N\geq 2$, $0<b<\min\{\tfrac{N}{2},2\}$, and $\tfrac{4-2b}{N}<\alpha<\tfrac{4-2b}{N-2}$. Suppose $u_0\in H^1(\R^N)$ obeys
\begin{equation}
E[u_0]^{s_c}M[u_0]^{1-s_c}<E[Q]^{s_c}M[Q]^{1-s_c}
\end{equation}
and
\begin{equation}
\|\nabla u_0\|_{L^2}^{s_c}\|u_0\|_{L^2}^{1-s_c}<\|\nabla Q\|_{L^2}^{s_c}\|Q\|_{L^2}^{1-s_c}.
\end{equation}
Then the corresponding solution $u$ to \eqref{INLS} is global-in-time and scatters. 
\end{theorem}

Theorem~\ref{T} is a direct extension of the results appearing in \cite{Campos, Dinh3, FG, FG2, MMZ, XZ19} (see also \cite{Dinh1} for scattering results in the defocusing case).  In particular, in \cite{Campos, FG, FG2} the same problem was considered with slightly stricter restrictions on the parameter $b$ and with the more significant restriction to \emph{radial} solutions.  In \cite{MMZ}, the authors introduced the essential new ingredient that allowed for the inclusion of non-radial initial conditions, albeit only in the case of the $3d$ cubic equation with $b\in(0,\tfrac12)$.  In the present work, we adapt this argument to treat the general (non-radial) intercritical setting and also extend the range of allowed parameters $(N,b,\alpha)$ beyond what has appeared in previous works. 

Our result fits into the broader context of sharp scattering thresholds for focusing intercritical nonlinear Schr\"odinger equations.  Such results were first established in the setting of the pure power-type NLS (see e.g. \cite{DHR, HR, Guevara, CFX, AkahoriNawa, DodMur}), while by now many extensions to related models are available (see e.g. \cite{KMVZ, FG, Zheng, LMM, Arora, KVZ, KVZ0}).  The general strategy in many of these works, as well as in the present paper, is to follow the `Kenig--Merle roadmap' of \cite{KM}, which entails reducing the problem of scattering for arbitrary solutions below the ground state threshold to the problem of precluding compact sub-threshold solutions (or `minimal blowup solutions').  The reduction to compact solutions relies heavily on concentration-compactness arguments, while the preclusion of such solutions is typically achieved using virial arguments.  For models like \eqref{INLS} involving broken symmetries (i.e. space translation in our setting), new difficulties arise in the construction of the minimal blowup solutions, specifically in the construction of certain nonlinear solutions to \eqref{INLS}.  In our case, we must show that we can construct scattering solutions associated to certain initial data involving translation parameters $x_n$ with $|x_n|\to\infty$.  Because of the broken translation symmetry, this cannot be achieved simply by incorporating a space-translation into a fixed nonlinear solution.  

In previous works on the inhomogeneous NLS (e.g. \cite{FG, FG2, Campos}), this issue was avoided by restricting to radial solutions---for such solutions, the translation parameters that one encounters may be chosen such that $x_n\equiv 0$.  In \cite{MMZ}, the authors adapted a strategy appearing in works such as \cite{KMVZZ2, KMVZ, LMM} (related in turn to many other recent works on dispersive equations with broken symmetries), demonstrating that in the setting of \eqref{INLS} one can use solutions to the \emph{linear} Schr\"odinger equation and a stability result for the nonlinear equation to construct solutions in the case $|x_n|\to\infty$; indeed, in the regime $|x|\to\infty$ one expects the effect of the nonlinearity $|x|^{-b}|u|^\alpha u$ to become very weak.  

In the present work, we utilize a similar argument and are thereby able to establish the existence of minimal blowup solutions (see Proposition~\ref{P:embed} and Proposition~\ref{P:exist}).  In this part of the argument, we also make some improvements to the allowed range of parameters $(N,b,\alpha)$ compared to the works \cite{Campos, FG, FG2}, which relies on the careful selection of suitable function spaces (see e.g. Lemma~\ref{L:NL}).  Once we have constructed our compact blowup solution, the rest of the argument follows from the virial argument essentially as in \cite{Campos, FG2}.  In fact, this also relies on the observation that profiles with $|x_n|\to\infty$ correspond to \emph{scattering} solutions to \eqref{INLS}.  In particular, this fact guarantees that such profiles do not appear when constructing minimal blowup solutions.  Accordingly, the `spatial center' $x(t)$ of our compact solution may be taken to be $x(t)\equiv 0$, which in turn facilitates the use of a localized virial argument. 

The rest of this paper is organized as follows:  In Section~\ref{S:notation} we first introduce some basic notation.  We also discuss the well-posedness theory and stability theory for \eqref{INLS}, as well as some variational analysis related to the ground state and the virial identity.  In Section~\ref{S:CC}, we discuss the concentration-compactness tools needed to construct minimal blowup solutions to \eqref{INLS}.  In Section~\ref{S:exist}, we show that if Theorem~\ref{T} fails, then we may construct a sub-threshold blowup solution to \eqref{INLS} with strong compactness properties.  Finally, in Section~\ref{S:proof} we describe how to use the virial argument to show that such a solution cannot exist, thus completing the proof of Theorem~\ref{T}.

\subsection*{Acknowledgements} M.C. was partially supported by Coordena\c c\~ao de Aperfei\c coamento de Pessoal de N\'ivel Superior - Brazil (CAPES). L.G.F. was partially supported by Coordena\c c\~ao de Aperfei\c coamento de Pessoal de N\'ivel Superior - Brazil (CAPES), Conselho Nacional de Desenvolvimento Cient\'ifico e Tecnol\'ogico - Brazil (CNPq) and Funda\c c\~ao de Amparo a Pesquisa do Estado de Minas Gerais - Brazil (FAPEMIG). C.M.G. was partially supported by Coordena\c c\~ao de Aperfei\c coamento de Pessoal de N\'ivel Superior - Brazil (CAPES).

\section{Notation and preliminaries}\label{S:notation}

We write $S^C=\R^{N}\backslash S$ for the complement of $S\subset\R^N$.  We write $a\lesssim b$ to denote $a\leq cb$ for some $c>0$, denoting dependence on various parameters with subscripts when necessary. If $a\lesssim b\lesssim a$, we write $a\sim b$. 

We utilize the standard Lebesgue spaces $L^p$, the mixed Lebesgue spaces $L_t^q L_x^r$, as well as the homogeneous and inhomogeneous Sobolev spaces $\dot H^{s,r}$ and $H^{s,r}$.  When $r=2$, we write $\dot H^{s,2}=\dot H^s$ and $H^{s,2}=H^s$.  If necessary, we use subscripts to specify which variable we are concerned with. We use $'$ to denote the H\"older dual. 

We also need the standard Littlewood--Paley projections $P_{\leq N}$, defined as Fourier multipliers, with the multiplier corresponding to a smooth cutoff to the
region $\{|\xi|\leq N\}$. In particular, we use the following Bernstein type inequality
\begin{equation}\label{BerIn}
\||\nabla|^sP_{\leq N} f\|_{L^2_x}\leq N^s\|f\|_{L^2_x}.
\end{equation}

\subsection{Well-posedness theory}

To discuss the well-posedness theory for \eqref{INLS}, we first recall the Strichartz estimates in the form that we will need them.  We call a pair of space-time exponents $(q,r)$ $\dot{H}^s$-admissible if they obey the scaling relation $\tfrac{2}{q}+\tfrac{N}{r}=\tfrac{N}{2}-s$ along with the constraints
\[
\begin{cases} \tfrac{2N}{N-2s}\leq r<\tfrac{2N}{N-2}, & N\geq 3, \\
\tfrac{2}{1-s} \leq r < \infty & N=2, \\
\tfrac{2}{1-2s}\leq r<\infty & N=1.
\end{cases}
\]
We call the pair $(q,r)$ $\dot{H}^{-s}$-admissible if they obey $\tfrac{2}{q}+\tfrac{N}{r}=\tfrac{N}{2}+s$ and the same constraints as above. 

Given $s\in\R$, we define $\mathcal{A}_s$ to be the set of $\dot H^{s}$-admissible pairs and introduce the Strichartz norm
\[
\|u\|_{S(\dot H^s)} = \sup_{(q,r)\in\mathcal{A}_s}\|u\|_{L_t^q L_x^r}.
\]
Here we assume space-time norms are taken over $\R\times\R^N$.  Restriction to a time interval $I$ and/or a subset $A\subset\R^N$ may be indicated by writing $S(\dot H^s(A);I)$.  

We define the dual Strichartz norm by
\[
\|g\|_{S'(\dot H^{-s})} = \inf_{(q,r)\in\mathcal{A}_{-s}}\|g\|_{L_t^{q'}L_x^{r'}},
\]
with similar notation for the restriction to subsets of $\R\times\R^N$. 

In the notation above, we may write the Strichartz estimates in the following form (see e.g. \cite{Cazenave, Kato, Foschi}):
\begin{align*}
\|e^{it\Delta}f\|_{S(\dot H^s)} & \lesssim \|f\|_{\dot H^{s}}, \\
\biggl\| \int_0^t e^{i(t-t')\Delta} g(t')\,dt'\biggr\|_{S(\dot H^s)} &\lesssim \|g\|_{S'(\dot H^{-s})}. 
\end{align*}

Strichartz estimates are one of the essential ingredients in the well-posedness theory for \eqref{INLS}, including small-data scattering, existence of wave operators, and stability theory.  Another key ingredient will be the following set of estimates on the nonlinearity. In particular we extend \cite[Lemma 4.7]{FG2} to all the intercritical regime and a larger range of the parameter $b$. Recall that $s_c=\tfrac{N}{2}-\tfrac{2-b}{\alpha}$ is the scaling invariant Sobolev index.
\begin{lemma}[Nonlinear estimates]\label{L:NL} Let $N\geq 2$, $0<b<\min\{2,\tfrac{N}{2}\}$ and $\tfrac{4-2b}{N}<\alpha<\tfrac{4-2b}{N-2}$. Then there exists $\theta\in(0,\alpha)$ sufficiently small that
\begin{itemize}
\item [(i)] $\left \||x|^{-b}|u|^\alpha v \right\|_{S'(\dot{H}^{-s_c})} \lesssim \| u\|^{\theta}_{L^\infty_tH^1_x}\|u\|^{\alpha-\theta}_{S(\dot{H}^{s_c})} \|v\|_{S(\dot{H}^{s_c})}$,
\item [(ii)] $\left\||x|^{-b}|u|^\alpha v \right\|_{S'(L^2)}\lesssim \| u\|^{\theta}_{L^\infty_tH^1_x}\|u\|^{\alpha-\theta}_{S(\dot{H}^{s_c})} \| v\|_{S(L^2)},
$
\item [(iii)] $\left\|\nabla (|x|^{-b}|u|^\alpha u)\right\|_{S'(L^2)}\lesssim\|u\|^{\alpha-\theta}_{S(\dot{H}^{s_c})}\left( \| u\|^{\theta}_{L^\infty_tH^1_x} \|\nabla u\|_{S(L^2)}+\| u\|^{1+\theta}_{L^\infty_tH^1_x}\right)$.
\end{itemize}
\end{lemma}

\begin{proof} 
In view of the singular factor $|x|^{-b}$ in the nonlinearity, we frequently divide our analysis in two regions. To this end we define $B\equiv B(0,1)=\{ x\in \mathbb{R}^N;|x|\leq 1\}$.

We introduce the parameters
\begin{equation}\label{PHsA1}  
\hat{q}=\tfrac{4\alpha(\alpha+2-\theta)}{\alpha(N\alpha+2b)-\theta(N\alpha-4+2b)},\quad \hat{r}=\tfrac{N\alpha(\alpha+2-\theta)}{\alpha(N-b)-\theta(2-b)},
\end{equation}
as well as
\begin{equation}\label{PHsA2}
\tilde{a}=\tfrac{2\alpha(\alpha+2-\theta)}{\alpha[N(\alpha+1-\theta)-2+2b]-(4-2b)(1-\theta)},\quad \hat{a}=\tfrac{2\alpha(\alpha+2-\theta)}{4-2b-(N-2)\alpha}.
\end{equation}
Then, for any small $\theta$, we have that $(\hat q,\hat r)$ is $L^2$-admissible, while $(\hat a, \hat r)$ is $\dot H^{s_c}$-admissible and $(\tilde a,\hat r)$ is $\dot H^{-s_c}$-admissible.  These exponents obey the scaling relations
\begin{equation}\label{LG1}
\tfrac{1}{\tilde{a}'}=\tfrac{\alpha-\theta}{\hat{a}}+\tfrac{1}{\hat{a}}\qtq{and}\tfrac{1}{\hat q'}=\tfrac{\alpha-\theta}{\hat{a}}+\tfrac{1}{\hat{q}}.   
\end{equation}

We first prove (i), using the pair $L_t^{\tilde a'} L_x^{\hat r'}$.  We let $r_1\in(\tfrac{1}{\theta},\infty)$ be a free parameter to be chosen below and define $\gamma$ so that 
\begin{equation}\label{LG1Hs2}
\tfrac{1}{\hat{r}'}=\tfrac{1}{\gamma}+\tfrac{1}{r_1}+\tfrac{\alpha-\theta}{\hat{r}}+\tfrac{1}{\hat{r}}.
\end{equation}
Then, with $A\in\{B,B^C\}$, we can first use H\"older's inequality to estimate 
 \begin{equation}\label{LG1Hs1}
\left \|  |x|^{-b}|u|^{\alpha} v \right\|_{L^{\hat{r}'}_x(A)} \lesssim \left\||x|^{-b}\right\|_{L^\gamma(A)}  \|u\|^{\theta}_{L^{\theta r_1}_x}   \|u\|^{\alpha-\theta}_{L_x^{\hat{r}}} \|v\|_{L^{\hat{r}}_x}.
\end{equation}
Using the scaling relations above, we derive 
 \begin{equation}\label{BRel}
\tfrac{N}{\gamma}-b=\tfrac{\theta(2-b)}{\alpha}-\tfrac{N}{r_1}.
\end{equation}
Thus, if $A=B$ and $N\geq 3$ we choose $r_1$ so that $\theta r_1=\tfrac{2N}{N-2}$, which therefore implies $|x|^{-b}\in L^\gamma(B)$.  If $N=2$ we just choose $\theta r_1>\tfrac{N\alpha}{2-b}$ so that the right hand side of \eqref{BRel} is positive.  On the other hand, if $A=B^C$, we choose $\theta r_1=2$, which similarly implies $|x|^{-b}\in L^\gamma(B^C)$. In both cases, we have by Sobolev embedding that $H^1\subset L^{\theta r_1}.$ Thus, continuing from above, we take the $L_t^{\tilde a'}$-norm, apply H\"older's inequality, and use the first scaling relation in \eqref{LG1} to obtain
\[
\left \|  |x|^{-b}|u|^{\alpha} v \right\|_{L_t^{\tilde a'} L_x^{\hat{r}'}(A)} \lesssim \|u\|_{L_t^\infty H_x^1}^\theta\|u\|_{L_t^{\hat a}L_x^{\hat r}}^{\alpha-\theta}\|v\|_{L_t^{\hat a}L_x^{\hat r}},
\]
which yields (i). The estimate of (ii) is similar.  In this case, we use the second scaling relation in \eqref{LG1} and derive
 \begin{equation}\label{LGHsii}
\left\|  |x|^{-b}|u|^{\alpha} v\right \|_{L_t^{\hat{q}'}L_x^{\hat{r}'}}\lesssim \|u\|^{\theta}_{L^\infty_tH^1_x}\|u\|^{\alpha-\theta}_{L_t^{\hat{a}} L_x^{\hat{r}}} \|v\|_{L^{\hat{q}}_tL^{\hat{r}}_x}.
\end{equation}

It remains to establish (iii). We must estimate two terms, one of the form $|x|^{-b}|u|^\alpha \nabla u $ and one of the form $|x|^{-b}[|x|^{-1}|u|^\alpha u]$. Using Hardy's inequality, we may estimate these two terms in essentially the same way, provided we choose function spaces in which Hardy's inequality may be applied.  For the case $N\geq 3$, such spaces were worked out explicitly in \cite[Lemma~2.7]{Campos}. Thus we will restrict our attention to the case $N=2$.

For the case $N=2$, we use the following exponents:
\begin{equation}\label{LGWP3}
\bar{a}=\tfrac{2(\alpha+1-\theta)}{1-s_c+\theta},\quad \bar{r}=\tfrac{2\alpha(\alpha+1-\theta)}{\alpha(1-b+s_c)+2-b-\theta(2-b+\alpha)},\quad \bar{q}=\tfrac{2(\alpha+1-\theta)}{1+\alpha s_c+\theta(1-s_c)}
\end{equation}
and
\begin{equation}\label{LGWP4}
a^*=\tfrac{2(\alpha-\theta)}{1+\theta},\quad r^*=\tfrac{2\alpha(\alpha-\theta)}{\alpha(1-b)-\theta(2-b+\alpha)},\quad q=\tfrac{2}{1-\theta},\quad r=\tfrac{2}{\theta},
\end{equation}
for $0<\theta\ll 1$.  In particular, for $\theta$ sufficiently small we have that $(\bar q,\bar r)$ is $L^2$-admissible and that $(\bar a,\bar r)$, $(a^*, r^*)$ are $\dot H^{s_c}$-admissible.  Furthermore, we have the scaling relations
\begin{equation}\label{GWP4}
\tfrac{1}{q'}=\tfrac{\alpha-\theta}{\bar{a}}+\tfrac{1}{\bar{q}}\qtq{and}(\alpha-\theta)q'=a^*.    
\end{equation}

As before, we take $A\in\{B,B^C\}$ and from \eqref{GWP4} we deduce 
\begin{align*}
\bigl\| & \nabla \left( |x|^{-b}|u|^{\alpha} u\right) \bigr\|_{L^{q'}_tL^{r'}_x(A)} \\
& \leq \left \|\left\||x|^{-b}\right\|_{L^\gamma(A)}\|u\|_{L_{x}^{\theta r_1}}^{\theta}\|u\|^{\alpha-\theta}_{L_x^{\bar{r}}} \|\nabla u \|_{L^{\bar{r}}_x}\right\|_{L^{q'}_t}\\
& \quad + \left\|\left\||x|^{-b-1}\right\|_{L^{\tilde\gamma}(A)}\|u\|_{L_{x}^{(\theta+1)p_1}}^{\theta+1} \|u\|^{\alpha-\theta}_{L_{x}^{r^*}}\right\|_{L^{q'}_t} \\
 & \lesssim  \|u\|_{L_t^\infty L_{x}^{\theta r_1}}^{\theta}\|u\|^{\alpha-\theta}_{L_t^{\bar{a}}L_x^{\bar{r}}} \|\nabla u \|_{L^{\bar{q}}_tL^{\bar{r}}_x}+\|u\|_{L_t^\infty L_{x}^{(\theta+1)p_1}}^{\theta+1} \|u\|^{\alpha-\theta}_{L_{t}^{a^*}L_{x}^{r^*}},
\end{align*}
where $r_1$ and $p_1$ will be chosen below and the choice of $\gamma,\tilde\gamma$ is then dictated by H\"older's inequality, i.e.
\[
\tfrac{1}{r'}=\tfrac{1}{\gamma}+\tfrac{1}{r_1}+\tfrac{\alpha-\theta}{\bar{r}}+\tfrac{1}{\bar{r}}=\tfrac{1}{\tilde\gamma}+\tfrac{1}{p_1}+\tfrac{\alpha-\theta}{r^*}.
\]
Using the definition of $r$, $\bar{r}$ and $r^*$ in \eqref{LGWP3}-\eqref{LGWP4} one has
\[
\tfrac{2}{\gamma}-b=\tfrac{\theta(2-b)}{\alpha}-\tfrac{2}{r_1}, \quad \tfrac{2}{\tilde \gamma}-b-1=\tfrac{\theta(2-b)}{\alpha}-\tfrac{2}{p_1}.
\]
If $A=B$, then we choose $\theta r_1$ and $\theta p_1\in (\tfrac{2\alpha}{2-b},\infty)$ (and observe that $\tfrac{2\alpha}{2-b}>2$). We then arrive at an acceptable estimate for the weights, as $\frac{2}{\gamma}-b>0$ and $\frac{2}{\tilde \gamma}-b-1>0$. If instead $A=B^C$, then we choose $\theta r_1, \theta p_1\in(2,\frac{2\alpha}{2-b})$ and again obtain a suitable estimate.  Using the Sobolev embedding $H^1\subset L^{\theta r_1} \cup L^{(\theta +1)p_1}$, we finally arrive at the estimate
\[
\left\| \nabla \left( |x|^{-b}|u|^{\alpha} u\right) \right\|_{L^{q'}_tL^{r'}_x}\leq  \|u\|_{L_t^\infty H_{x}^1}^{\theta}\|u\|^{\alpha-\theta}_{L_t^{\bar{a}}L_x^{\bar{r}}} \|\nabla u \|_{L^{\bar{q}}_tL^{\bar{r}}_x}+\|u\|_{{L_t^\infty H_{x}^1}}^{\theta+1} \|u\|^{\alpha-\theta}_{L_{t}^{a^*}L_{x}^{r^*}},   
\]
which completes the proof of (iii).\end{proof} 

With these nonlinear estimates in place, standard arguments utilizing Strichartz estimates suffice to establish the following result (see e.g. \cite[Proposition 1.4 and Proposition 4.15]{FG2})
\begin{proposition}[Well-posedness]\label{P:LWP} Let $N\geq 2$, $0<b<\min\{2,\tfrac{N}{2}\}$ and $\tfrac{4-2b}{N}<\alpha<\tfrac{4-2b}{N-2}$. 
\begin{itemize}
\item[(i)] For any $u_0\in H^1$, there exists a local-in-time solution to \eqref{INLS}.  The time of existence depends on the $H^1$-norm of $u_0$.  In particular, any solution that remains uniformly bounded in $H^1$ throughout its lifespan exists for all time.  If $u_0$ is sufficiently small in $H^1$, then the corresponding solution to \eqref{INLS} is global-in-time and scatters.
\item[(ii)] A solution $u$ to \eqref{INLS} may be extended as long as its $S(\dot H^{s_c})$-norm does not blow up.  A global solution that remains bounded in $H^1$ and has finite $S(\dot H^{s_c})$-norm scatters in both time directions. 
\item[(iii)] For any $\psi\in H^1$, satisfying 
\begin{equation}\label{finalstate-subthreshold}
\|\psi\|_{L^2}^{2(1-s_c)}\|\nabla\psi\|_{L^2}^{2s_c} < 2^{s_c}M[Q]^{1-s_c}E[Q]^{s_c},
\end{equation}
there exists a global and uniformly bounded in $H^1$ solution $u$ to \eqref{INLS} that scatters to $\psi$ as $t\to\infty$, that is
\[
\lim_{t\to\infty}\|u(t) - e^{it\Delta}\psi\|_{H^1(\R^N)}=0.
\] 
\end{itemize}
\end{proposition}

\begin{proof} Item (i) is the standard well-posedness result, while item (ii) gives the standard scattering criterion. Item (iii) gives the existence of wave operators.  To prove it, one firstly solves the local problem around $t=\infty$. The condition \eqref{finalstate-subthreshold} guarantees (using Lemma~\ref{L:coercive} below) that the solution constructed lies below the ground state and in particular has uniformly bounded $H^1$-norm throughout its lifespan, which in turn allows the solution to be extended for all time. 
\end{proof}

In addition, we will need the standard stability result for equation \eqref{INLS} (see e.g. \cite[Proposition 4.14]{FG2}). 

\begin{proposition}[Stability]\label{stability} Let $N\geq 2$, $0<b<\min\{2,\tfrac{N}{2}\}$, and $\tfrac{4-2b}{N}<\alpha<\tfrac{4-2b}{N-2}$. Let $0\in I\subseteq \mathbb{R}$ and $\tilde u:I\times \mathbb{R}^N\to\C$ be a solution to 
\begin{equation*}
i\partial_t \tilde{u} +\Delta \tilde{u} + |x|^{-b} |\tilde{u}|^\alpha \tilde{u} =e,
\end{equation*}  
with initial data $\tilde{u}_0\in H^1(\mathbb{R}^N)$ satisfying \begin{equation*}
\sup_{t\in I}  \|\tilde{u}\|_{H^1_x}\leq M \;\; \textnormal{and}\;\; \|\tilde{u}\|_{S(\dot{H}^{s_c}; I)}\leq L.
\end{equation*}
for some positive constants $M,L$.

Let $u_0\in H^1(\mathbb{R}^N)$ such that 
\begin{equation*}
\|u_0-\tilde{u}_0\|_{H^1}\leq M'\qtq{and} \|e^{it\Delta}[u_0-\tilde{u}_0]\|_{S(\dot{H}^{s_c}; I)}\leq \varepsilon,
\end{equation*}
for some positive constant $M'$ and some $0<\varepsilon<\varepsilon_1=\varepsilon_1(M,M',L)$.

Suppose further that
\begin{equation*}
\|e\|_{S'(L^2; I)}+\|\nabla e\|_{S'(L^2; I)}+ \|e\|_{S'(\dot{H}^{-s_c}; I)}\leq \varepsilon.
\end{equation*}
\indent Then, there exists a unique solution $u:I\times\R^3\to\C$ with $u|_{t=0}=u_0$ obeying	
\begin{align*}
\|u-\tilde{u}\|_{S(\dot{H}^{s_c}; I)}&\lesssim_{M,M',L}\varepsilon,\\
\|u\|_{S(\dot{H}^{s_c}; I)} +\|u\|_{S(L^2; I)}+\|\nabla u\|_{S(L^2; I)}& \lesssim_{M,M',L} 1.
\end{align*} 
\end{proposition}

\subsection{Variational analysis} In this section we collect a few results from \cite{Farah} related to the ground state $Q$, which is the unique radial, nonnegative, decaying solution to 
\[
-\Delta Q - Q + |x|^{-b}Q^{\alpha+1} = 0
\]
(see \cite{GENSTU, Yanagida}). This solution may be constructed as an optimizer to the Gagliardo--Nirenberg inequality
\[
\|\,|x|^{-b}|u|^{\alpha+2}\|_{L^1} \leq C_{GN} \|\nabla u\|_{L^2}^{\frac{N\alpha+2b}{2}}\|u\|_{L^2}^{\frac{4-2b-\alpha(N-2)}{2}}.
\]
In particular, the sharp constant can be expressed in terms of the parameters $(N,\alpha,b)$ and the $L^2$ and $\dot H^1$ norms of $Q$.  Using this, one can obtain the following coercivity result (see e.g. \cite[Lemma 4.2]{FG2}).  
\begin{lemma}[Coercivity]\label{L:coercive} Let $v\in H^1$ obey
\begin{equation}\label{coercivity-assumption}
M[v]^{1-s_c}E[v]^{s_c}<M[Q]^{1-s_c}E[Q]^{s_c} \qtq{and} \|u\|_{L^2}^{1-s_c}\|\nabla u\|_{L^2}^{s_c}\leq \|Q\|_{L^2}^{1-s_c}\|\nabla Q\|_{L^2}^{s_c}.
\end{equation}
Then we have the following:
\begin{itemize}
\item[(i)] $E[v] \sim \|\nabla v\|_{L^2}^2$, 
\item[(ii)] $\|v\|_{L^2}^{1-s_c}\|\nabla v\|_{L^2}^{s_c} <(1-\delta)\|Q\|_{L^2}^{1-s_c}\|\nabla Q\|_{L^2}^{s_c}$ for some $\delta>0$, 
\item[(iii)]  $8\|\nabla v\|_{L^2}^2-\tfrac{4(N\alpha+2b)}{\alpha+2}\|\,|x|^{-b} |v|^{\alpha+2}\|_{L^1} >\delta\|\nabla v\|_{L^2}^2$ for some $\delta>0$. 
\end{itemize}
\end{lemma}

In particular, item (i) shows the coercivity of the energy under the sub-threshold assumption, while (ii) shows a quantitative improvement to the second assumption in \eqref{coercivity-assumption}.  Finally, (iii) shows that below the ground state the quantity appearing in the virial computation is quantitatively positive (see below).

\subsection{Virial identities}\label{S:virial} In this section, we recall the virial identity obtained in \cite[Proposition~7.2]{FG2}. In particular, for a solution $u$ to \eqref{INLS}, we define
\[
z_R(t) = \int_{\R^N} R^2\phi(\tfrac{x}{R})|u(t,x)|^2\,dx,
\]
where $\phi$ is a smooth function.  Then 
\[
z_R'(t) =2R \Im \int_{\R^N} \nabla \phi(\tfrac{x}{R})\cdot\nabla u(x) \bar u(x)\,dx,
\]
and
\begin{align*}
z_R''(t)& = 4\Re\int \partial_k u \partial_j \bar u \partial_{jk}\phi(\tfrac{x}{R}) - \tfrac1{R^2} |u|^2 (\Delta^2\phi)(\tfrac{x}{R})\,dx \\
& \quad -\tfrac{2\alpha}{\alpha+2}\int |x|^{-b}|u|^{\alpha+2}(\Delta \phi)(\tfrac{x}{R}) + \tfrac{4R}{\alpha+2} \nabla(|x|^{-b})\cdot\nabla\phi(\tfrac{x}{R})|u|^{\alpha+2}\,dx,
\end{align*}
where repeated indices are summed.

The standard virial identity corresponds to choosing $\phi(x)=|x|^2$, in which case the second derivative reduces to the quantity appearing in Lemma~\ref{L:coercive}(iii).  In particular, this term has a good sign for solutions below the ground state.  However, in this case, the virial quantity $z_R'$ cannot be controlled uniformly in time.  For the localized virial argument, one takes $\phi$ to be a smooth, compactly supported supported function equal to $|x|^2$ for $|x|\leq 1$ and equal to zero for $|x|>2$.  Then in the second derivative term, one still obtains the quantity in Lemma~\ref{L:coercive}(iii), up to errors that are controlled by
\begin{equation}\label{errors}
\int_{|x|>R} |\nabla u(t,x)|^2 + \tfrac{1}{R^2}|u(t,x)|^2+\tfrac{1}{R^b}|u(t,x)|^{\alpha+2}\,dx. 
\end{equation}
Such error terms can be controlled uniformly in time for solutions with pre-compact orbit in $H^1$.  With this localization, one also has the bound
\begin{equation}\label{zRbd}
|z_R'(t)| \lesssim R\|u(t)\|_{H^1}^2,
\end{equation}
which (for fixed $R$) has the potential to be bounded uniformly in $t$. 
\subsection{Concentration compactness}\label{S:CC}

An essential ingredient in the construction of minimal blowup solutions is the following linear profile decomposition for $H^1$-bounded sequences adapted to the Strichartz estimates discussed above.  The version we need can be found, for example, in \cite[Theorem 5.1]{CFX} (see also \cite[Theorem 1.10]{Shao}). 

\begin{proposition}[Linear profile decomposition]\label{profile}
Let $\{\phi_n\}$ be a bounded sequence in $H^1(\R^N)$.  There exists $M^*\in\mathbb{N}\cup\{\infty\}$ so that for each $1 \leq j \leq M\leq M^*$ there exists a profile $\psi^j$ in $H^1(\R^N)$, a sequence $t^j_n$ of time shifts, a sequence $x^j_n$ of space shifts, and a sequence $W_n^M(x)$ of remainders in $H^1(\R^N)$, such that
\begin{align}
\phi_n(x)=\sum_{j=1}^{M}e^{-it^j_n\Delta}\psi^j(x-x^j_n)+W^{M}_n(x)
\end{align}
along a subsequence, with the following properties:
\begin{itemize}
\item Asymptotic orthogonality of the parameters: for $1\leq k\neq j\leq M$,
\begin{align}\label{orthogtx}
\lim_{n\to+\infty}|t^j_n-t^k_n|+|x^j_n-x^k_n|=\infty.
\end{align}
\item Asymptotic vanishing for the remainders: 
\begin{align}\label{restzero}
\lim_{M\to+\infty}\left(\lim_{n\to +\infty}\|e^{it\Delta}W^M_n\|_{S(\dot H^{s_c})}\right)=0.
\end{align}
\item Asymptotic mass/energy decoupling: for $M\in \mathbb{N}$ and any $s\in [0,1]$, we have
\begin{align}\label{pythagoreanexpansion}
\|\phi_n\|_{\dot H^s}^2=\sum_{j=1}^{M}\|\psi^j\|_{\dot H^s}^2+\|W^M_n\|_{\dot H^s}^2+o_n(1) \qtq{as}n\to\infty.
\end{align}
\end{itemize}
Finally, we may assume either $t_n^j\equiv 0$ or $t_n^j\to\pm\infty$, and either $x_n^j\equiv 0$ or $|x_n^j|\to\infty$.
\end{proposition}

\section{Construction of minimal blowup solutions}\label{S:exist}

In this section, we show that if Theorem~\ref{T} fails, we may find a minimal blowup solution below the ground state threshold that obeys strong compactness properties.  A similar result was obtained in \cite{FG2} for radial solutions and under more stringent restrictions on the parameters $(N,b,\alpha)$, while the non-radial case was addressed in \cite{MMZ} for the $3d$ cubic case. The key ingredients are the linear profile decomposition (Proposition~\ref{profile}), stability theory (Proposition~\ref{stability}), and the construction of scattering solutions living far from the origin (Proposition~\ref{P:embed} below). In the next result we construct the minimal blowup solution. In particular, we extend \cite[Proposition 6.1]{FG2} to the non-radial setting for all the intercritical regime.

\begin{proposition}[Existence of a minimal blowup solution]\label{P:exist} Suppose Theorem~\ref{T} fails for some choice of $(N,b,\alpha)$ as in the statement of the theorem.  Then there exists a function $u_{c,0}\in H^1$ so that the corresponding solution $u_c$ to \eqref{INLS} is global, uniformly bounded in $H^1$, and obeys the following: 
\begin{itemize}
\item [(i)] $M[u_c]=1$,
\item [(ii)] $E[u_c]<E[Q]$,
\item [(iii)] $\|  \nabla u_{c,0} \|_{L^2}^{s_c} \|u_{c,0}\|_{L^2}^{1-s_c}<\|\nabla Q \|_{L^2}^{s_c} \|Q\|_{L^2}^{1-s_c}$,
\item [(iv)] $\|u_{c}\|_{S(\dot{H}^{s_c};\mathbb{R}_+)}=\|u_{c}\|_{S(\dot{H}^{s_c};\mathbb{R}_-)}=+\infty$.
\end{itemize}
Furthermore, the orbit $\{u_c(t):t\in\R\}$  is pre-compact in $H^1$. 
\end{proposition}

\begin{proof}
For each $\delta > 0$, we define the set $A_\delta$ of all $u_0\in H^1(\mathbb{R}^N)$ satisfying
\[
E[u_0]^{s_c}M[u_0]^{1-s_c}<\delta\qtq{and}\|\nabla u_0\|^{s_c}_{L^2}\| u_0\|^{1-s_c}_{L^2}<\|\nabla Q\|^{s_c}_{L^2}\| Q\|^{1-s_c}_{L^2}.
\]
We define $\delta_c$ to be the supremum of $\delta>0$ such that data in $A_\delta$ leads to global scattering solutions to \eqref{INLS} with finite space-time norm.  The assumption that Theorem~\ref{T} fails for a choice of $(N,b,\alpha)$ is then equivalent to the statement that $\delta_c<M[Q]^{1-s_c}E[Q]^{s_c}.$

We may now take a sequence of solutions $u_n$ to \eqref{INLS} with $H^1$ initial data $u_{n,0}$ obeying the following: 
\begin{align}
&\|u_{n,0}\|_{L^2} \equiv 1,\label{mass=1} \\
& \|\nabla u_{n,0}\|^{s_c}_{L^2} <  \|\nabla Q\|^{s_c}_{L^2}\|Q\|^{1-s_c}_{L^2},\label{PCS1}   \\
& E[u_n] \to \delta_c^{\frac{1}{s_c}}\qtq{as} n \rightarrow \infty, \label{PCS2} \\
&\lim_{n\to\infty} \|u_n\|_{S(\dot H^{s_c};\R_+)}  = \lim_{n\to\infty} \|u_n\|_{S(\dot H^{s_c};\R_-)} = \infty. \label{un}
\end{align}

We will now apply the linear profile decomposition (Proposition~\ref{profile}) to the sequence $u_{n,0}$ and establish the following facts: 
\begin{itemize}
\item[(a)] there is a single profile $\psi$ present in the decomposition; 
\item[(b)] the time shifts obey $t_n\equiv 0$; 
\item[(c)] the space shifts obey $x_n\equiv 0$; and 
\item[(d)] the error $W_n$ converges to zero strongly in $H^1$.
\end{itemize}

In particular, items (a)--(d) imply that the sequence $u_{n,0}$ converges strongly to some limit $u_{c,0}$ (after passage to an appropriate subsequence). The solution $u_c$ to \eqref{INLS} with initial data $u_{c,0}$ will then obey all of the conditions appearing in Proposition~\ref{P:exist}.  Indeed, items (i)---(iii) follow from the strong $H^1$ convergence, while (iv) follows from the stability result (Proposition~\ref{stability}).  To establish the pre-compactness of the orbit of $u_c$, one must prove convergence (along a subsequence) for $\{u_c(t_n)\}$ for an arbitrary sequence of times $t_n$.  To this end, one simply repeats arguments of the present proof to the sequence $u_c(t_n)$ (in place of the sequence $u_{n,0})$; indeed, this sequence and obeys \eqref{mass=1}, \eqref{PCS1}, \eqref{PCS2}, and \eqref{un} above. 

It therefore remains to establish items (a)--(d).  We apply Proposition~\ref{profile} to obtain 
 \begin{equation}\label{PCS3}
 u_{n,0}(x)=\sum_{j=1}^{M}e^{-it_n^j\Delta} \psi^j(x-x_n^j)+W_n^M(x),
 \end{equation} 
where the shifts, profiles, and remainders obey all of the conditions stated in Proposition~\ref{profile}.

Before proceeding to the proof of items (a)--(d), let us collect a few facts about the profiles appearing in the decomposition above.  First, using the Pythagorean expansion \eqref{pythagoreanexpansion}, we can obtain 
\begin{align}
\sum_{j=1}^{M}\|\psi^j\|^2_{L^2}+\limsup_{n\to\infty}\|W_n^M\|_{L^2}^2\leq 1, \label{PCS41}
\end{align}
for all $M$.  Similarly, using Lemma~\ref{L:coercive} as well, we can deduce that 
\[
E[e^{-it_n^j\Delta}\psi^j]\geq 0 \qtq{and} E[W_n^M]\geq 0,
\]
for all $j$ and all $n$ large, along with the fact that
\begin{equation}\label{PCS8}
\limsup_{n\rightarrow+\infty}\biggl[\sum_{j=1}^{M}E[e^{-it_n^j\Delta}\psi^j]+E[W_n^M]\biggr]=\delta_c^{\frac{1}{s_c}}.
\end{equation}
Moreover we have that for both the profiles and the remainder, the energy is comparable to the square of the $\dot H^1$-norm (see Lemma~\ref{L:coercive}).

\textbf{Item (a).} We turn to item (a) and suppose towards a contradiction that more than one profile appears in the decomposition above.  (Note that there must be at least one profile, for otherwise an application of the stability result Proposition~\ref{stability} with the approximate solution $e^{it\Delta}W_n^M$ would imply global space-time bounds for the solutions $u_n$, contradicting \eqref{un}.)  In this case, we can construct scattering solutions to \eqref{INLS} corresponding to each profile, as we now explain.

First, if $x_n^j\equiv 0$ and $t_n^j\equiv 0$, then we take $v^j$ to be the solution to \eqref{INLS} with initial data $\psi^j$.  This solution scatters due to the fact that it is below the critical scattering threshold (see \eqref{PCS3} and \eqref{PCS8}). If instead $x_n^j\equiv 0$ and $t_n^j\to\pm\infty$, we take $v^j$ to be the solution that scatters to $e^{it\Delta}\psi^j$ as $t\to\pm\infty$ (cf. Proposition~\ref{P:LWP}(iii)). In both cases we then set $v_n^j(t,x)=v^j(t+t_n^j,x)$.  Finally, if $|x_n^j|\to\infty$, then we must appeal to Proposition~\ref{P:embed}, proved below, to construct a global scattering solution $v_n^j$ to \eqref{INLS} obeying $v_n^j(0,x)=e^{-it_n^j\Delta}\psi(x-x_n^j)$. 

We now define a sequence of approximate solutions to \eqref{INLS} via
\[
\tilde u_n^M = \sum_{j=1}^M v_n^j.
\]
We now claim the following: 
\begin{itemize}
\item[1.] We have asymptotic agreement of the initial data:
\begin{equation}\label{CSR2}
 \limsup_{M\rightarrow \infty} \left[\limsup_{n\rightarrow \infty}  \|e^{it\Delta}[u_n(0)-\tilde{u}_{n}^M(0)]\|_{S(\dot{H}^{s_c})}\right]=0
\end{equation}
\item[2.] The functions $\tilde u_n^M$ obey uniform space-time bounds: 
There exist $L>0$ and $S>0$ independent of $M$ such that for any $M$, there exists $n_1=n_1(M)$ such that
 \begin{equation}\label{claim1}
 n>n_1\implies   \|\tilde{u}_n^M\|_{S(\dot{H}^{s_c})}\leq L\qtq{and}\|\tilde{u}_n^M\|_{L^\infty_tH^1_x}\leq S.
 \end{equation}
\item[3.] The functions $\tilde u_n^M$ are good approximate solutions to \eqref{INLS}.  That is, defining $f(z)=|z|^\alpha z$ and 
\begin{align*}
e_n^M & = (i\partial_t+\Delta)\tilde u_n^M + |x|^{-b}f(\tilde u_n^M) = |x|^{-b}\biggl[ f\bigl( \sum_{j=1}^M v_n^j\bigr) - \sum_{j=1}^M f(v_n^j) \biggr],
\end{align*}
we have the following: for each $M$ and $\varepsilon>0$, there exists $n_0=n_0(M,\varepsilon)$ such that
 \begin{equation}\label{claim2}
n>n_0\implies   \|e_n^M\|_{S'(\dot{H}^{-s_c})}+\|e_n^M\|_{S'(L^2)}+\|\nabla e_n^M\|_{S'(L^2)}\leq\varepsilon.
\end{equation}
\end{itemize}
Once we have established these three claims, we may apply the stability result (Proposition~\ref{stability}) to deduce that the solutions $u_n$ inherit the space-time bounds of the approximate solutions $\tilde u_n^M$, thus leading to a contradiction to \eqref{un} and completing the proof of item (a). 

For the third claim, it is enough to quote \cite[Proposition~6.1]{FG2}.  The proof relies essentially on the orthogonality of the functions $v_n^j$ in the form of \eqref{orthogtx}, which is used in conjunction with approximation by $C_c^\infty(\R\times\R^3)$ functions.  In particular, one applies pointwise estimates to the difference appearing in $e_n^M$ (as well as to the gradient) and then utilizes both the orthogonality condition as well as the global space-time bounds obeyed by the individual $v_n^j$. 

We therefore turn to the first claim and prove \eqref{CSR2}. By construction, we have
\[
u_n(0,x)-\tilde u_n^M(0,x) = \sum_{j=1}^M \bigl[e^{-it_n^j\Delta}\psi^j(x-x_n^j) -v^j(t_n^j,x)\bigr] + W_n^M.
\]
In particular, the summands are either identically zero or converge to zero in $H^1$ as $n\to\infty$, while the free evolution of $W_n^M$ tends to zero in $S(\dot H^{s_c}$) as $n,M\to\infty$.  Thus, \eqref{CSR2} holds. 

We turn to the second claim and prove \eqref{claim1}.  We follow the usual approach (namely, we use the fact that each profile obeys global space-time bounds and exploit the orthogonality in \eqref{orthogtx} to sum); however, we will provide some detail here, as it is in this step that a careful choice of function spaces extends the range of parameters $(N,\alpha,b)$ compared to previous works on \eqref{INLS}. 

First, using \eqref{PCS3} and \eqref{PCS41} together with the small-data theory for \eqref{INLS} (see Proposition~\ref{P:LWP}), we may obtain
\begin{equation}\label{CS1}
\sum_{j\geq M_0} \bigl[ \|v_n^j\|_{S(L^2)}^2 + \| \nabla v_n^j\|_{S(L^2)}^2\bigr] \lesssim 1
\end{equation}
for some sufficiently large $M_0$ and for all large $n$. On the other hand, utilizing the orthogonality condition \eqref{orthogtx}, one can establish
\[
\sup_{t\in\R}\bigl| \langle v_n^j,v_n^k\rangle_{H^1} \bigr| \to 0 \qtq{as}n\to\infty
\]
for $j\neq k$  (see e.g. \cite[Corollary 4.4]{CFX}).  Using the above, we can establish 
 \begin{equation}\label{claim12}
 \sup_{t\in\mathbb{R}}\|\tilde{u}_n^M\|^2_{H^1_x}\leq S \qtq{for all} n>n_1(M)
\end{equation}
for some $S>0$ independent of $M$. 

We will next find space-time norms in which we can estimate $\tilde u_n^M$.  We treat separately the cases $N=2$ and $N\geq 3$.

First suppose $N\geq 3$.  We will estimate $\tilde u_n^M$ in the $\dot{H}^{s_c}$-admissible space $L_t^{\hat a}L_x^{\hat r}$, where we recall the pair $(\hat a,\hat r)$ from the proof of Lemma~\ref{L:NL} (see \eqref{PHsA1} and \eqref{PHsA2}).  To this end, we first define
\[
\hat{p}=\tfrac{2N(\alpha+2-\theta)}{N(\alpha+2-\theta)-4(1-s_c)}
\]
for some $0<\theta\ll 1$ to be specified more precisely below. One has that $(\hat{a},\hat{p})$ is $L^2$-admissible and $s_c=\frac{N}{\hat{p}}-\frac{N}{\hat{r}}$. Moreover, $\hat{p}<\frac{N}{s_c}$ for $\theta$ small enough and $2<\hat{p}<\frac{2N}{N-2}$. Hence, applying Sobolev embedding we get
\begin{equation}\label{estimate-unM}
\begin{aligned}
 \|&\tilde{u}_n^M\|^2_{L_t^{\hat{a}}L_x^{\hat{r}}} \\ 
 & \lesssim\biggl\|\biggl(\sum_{j=1}^M |\nabla|^{s_c} v_n^j \biggr)^2\biggr\|_{L_t^{\frac{\hat{a}}{2}}L_x^{\frac{\hat{p}}{2}}}\\
 &\lesssim\sum_{j=1}^{M_0} \left\| |\nabla|^{s_c} v_n^j \right\|^2_{L_t^{\hat{a}}L_x^{\hat{p}}}+\sum_{j=M_0}^{M} \left\| |\nabla|^{s_c} v_n^j \right\|^2_{L_t^{\hat{a}}L_x^{\hat{p}}}+\sum_{j\neq k} \left\| [|\nabla|^{s_c} v_n^j][ |\nabla|^{s_c}v_n^k] \right\|_{L_t^{\frac{\hat{a}}{2}}L_x^{\frac{\hat{p}}{2}}}\\
 \end{aligned}
 \end{equation}
The first term is bounded by some constant depending on $M_0$, where as the second term is bounded independent of $M$ in light of \eqref{CS1}.  On the other hand, using the orthogonality \eqref{orthogtx} and approximation by $C_c^\infty(\R\times\R^N)$ functions, we see that the final term on the right-hand side above tends to zero as $n\to\infty$.

Thus we conclude that $\tilde u_n^M$ obeys good bounds in the specific norm $L_t^{\hat{a}} L_x^{\hat r}$.  To extend to arbitrary $S(\dot H^{s_c})$-admissible spaces, we make use of Strichartz estimates, estimating as in Lemma~\ref{L:NL} and recalling that $\tilde u_n^M$ is an approximate solution to \eqref{INLS} with error term obeying the estimates appearing in \eqref{claim2}. 

We next consider the case $N=2$. We introduce the parameters
\begin{equation}\label{C2N21}
a=\tfrac{2\alpha(\alpha+1-\theta)}{2-b+\varepsilon},\quad\,r=\tfrac{2\alpha(\alpha+1-\theta)}{(2-b)(\alpha-\theta)-\varepsilon}, 
\end{equation}
as well as
\begin{equation}\label{C2N22}
\bar{a}=\tfrac{2\alpha}{2\alpha-(2-b)-\varepsilon}\quad\bar{r}=\tfrac{2\alpha}{\varepsilon},
\end{equation}
where $\theta\in (0,\alpha)$ and $\varepsilon>0$ are chosen sufficiently small so that $a>4$ (this is possible since $\alpha>2-b$ and $b<1$). A direct computation shows that $(a,r)$ is $\dot{H}^{s_c}$-admissible and $(\bar{a},\bar{r})$ is $\dot{H}^{-s_c}$ admissible.

By interpolation, we have
\[
\|\tilde{u}_n^M\|_{L^a_tL^r_x}\leq \|\tilde{u}_n^M\|^{1-\frac{4}{a}}_{L^\infty_tL^p_x}\|\tilde{u}_n^M\|^{\frac{4}{a}}_{L_{t,x}^4},
\]
where
\[
p=\tfrac{2\alpha(\alpha+1-\theta)-4(2-b+\varepsilon)}{(2-b)(\alpha-1-\theta)-2\varepsilon}.
\]
We first observe that we may control $\tilde u_n^M$ in $L_{t,x}^4$ by estimating as we did in \eqref{estimate-unM} above.  Next, we note that $p>2$ for small enough $\theta$ and $\varepsilon$, which follows from $\alpha>2-b$ and $b<1$.
Thus by the Sobolev embedding $H^1 \hookrightarrow L^p$, we obtain control over $\tilde u_n^M$ in $L_t^a L_x^r$.  

As before, to extend to arbitrary $\dot H^{s_c}$ admissible pairs we make use of Strichartz estimates and the fact that $\tilde u_n^M$ is an approximate solution to \eqref{INLS} with errors obeying \eqref{claim2}. In particular, we need a suitable nonlinear estimate, which we obtain as follows.  We let $r_1\in(\tfrac{1}{\theta},\infty)$ be a free parameter to be chosen below and define $\gamma$ so that
\[
\tfrac{1}{\bar{r}'} = \tfrac{1}{\gamma}+\tfrac{1}{r_1} + \tfrac{\alpha+1-\theta}{r},
\]
where $\bar r$ and $r$ are given by relation \eqref{C2N21} and \eqref{C2N22}.
We may then estimate as follows: for $A\in\{B,B^C\}$, we have by H\"older's inequality 
\begin{align*}
\left\|  |x|^{-b}  |\tilde{u}_n^M|^\alpha \tilde{u}_n^M \right\|_{L^{\bar{a}'}_tL_x^{\bar{r}'}(A)}
& \lesssim  \| |x|^{-b}\|_{L^\gamma(A)} \|\tilde{u}_n\|^{\theta}_{L_t^\infty L_x^{\theta r_1}}   \|\tilde{u}_n\|^{\alpha+1-\theta}_{L_t^{a}L_x^{r}},
\end{align*}
since $(\alpha+1-\theta)\bar{a}'=a$.
Using the scaling relations, we see that
\[
\tfrac{2}{\gamma}-b = \tfrac{\theta(2-b)}{\alpha}-\tfrac{2}{r_1}. 
\]
In particular, if $A=B$ we choose $r_1$ so that $\theta r_1 > \tfrac{2\alpha}{2-b}$, while if $A=B^c$ we choose $r_1$ so that $\theta r_1\in(2,\tfrac{2\alpha}{2-b})$. With this choice, we have (after another application of Sobolev embedding) suitable control over the quantity above, which completes the proof of \eqref{claim1} in the case $N=2$. 

With \eqref{claim1} in place, we have now completed the proof of item (a), that is, the existence of a single profile in the decomposition \eqref{PCS3}.  We now turn to items (b)--(d).

\textbf{Items (b)--(d).} We have reduced the decomposition \eqref{PCS3} to one of the form
\[
u_{n,0}(x)= e^{-it_n\Delta}\psi(x-x_n)+W_n(x).
\]
To see that the space shifts must obey $x_n\equiv 0$, we note that if $|x_n|\to\infty$ then Proposition~\ref{P:embed} yields global scattering solutions $v_n$ to \eqref{INLS} with $v_n(0)=e^{-it_n\Delta}\psi(x-x_n)$.  Applying the stability result (Proposition~\ref{stability}), this implies uniform space-time bounds for the solutions $u_n$, contradicting \eqref{un}. To see that the time shifts must obey $t_n\equiv 0$, we note that if $|t_n|\to\infty$ then the functions $v_n(t)= e^{i(t-t_n)\Delta}\psi$ define good approximate solutions obeying global space-time bounds for $n$ large.  In particular, an application of Proposition~\ref{stability} would again yield uniform space-time bounds for the $u_n$, contradicting \eqref{un}.  Finally, note that if $M[\psi]<1$ or $E[\psi]<\delta^{\frac{1}{s_c}}$, then by the criticality of $\delta_c$ and stability theory once again, we would have that the solutions $u_n$ obey uniform space-time bounds.  It follows that $W_n$ must converge to zero strongly in $H^1$.  This completes the proof of items (b)--(d) and hence the proof of Proposition~\ref{P:exist}. 
\end{proof}

It remains to show how to construct scattering solutions to \eqref{INLS} corresponding to profiles living far from the origin.  In particular, we will extend the result of \cite{MMZ} beyond the special case of the $3d$ cubic equation (and even for that case, we extend the range of the parameter $b$ from $b\in(0,\tfrac12)$ to $b\in(0,1)$).  Our improvement stems largely from the fact that we have found suitable spaces in which to carry out the nonlinear estimates (see Lemma~\ref{L:NL}).

\begin{proposition}\label{P:embed} Suppose $N\geq 2$, $0<b<\min\{\tfrac{N}{2},2\}$, and $\tfrac{4-2b}{N}<\alpha<\tfrac{4-2b}{N-2}$.  Let $\psi \in H^1(\mathbb{R}^N)$. Suppose that $t_n\equiv 0$ or $t_n\to\pm\infty$ and that $|x_n|\to\infty$. Then for all $n$ sufficiently large, there exists a global solution
$v_n$ to \eqref{INLS} with
\[
v_n(0)=\psi_n:=e^{it_n\Delta} \psi(x-x_n)
\]
that scatters in $H^1$ and obeys
\[
\|v_n\|_{S(\dot{H}^{s_c})}+\|v_n\|_{S(L^2)}+\|\nabla v_n\|_{S(L^2)}\leq C
\]
for some $C=C(\|\psi\|_{H^1})$. Moreover, for any $\varepsilon > 0$, there exists $K>0$ and $\phi \in C_c^\infty(\mathbb{R}\times \mathbb{R}^N)$ such that
\begin{equation}\label{aprox}
 \|v_n - \phi(\cdot+ t_n, \cdot- x_n)\|_{S(\dot{H}^{s_c})} < \varepsilon\;\;\textnormal{for}\;\; n \geq K.   
\end{equation}
\end{proposition}

\begin{proof} We begin by constructing approximate solutions to \eqref{INLS}. For each $n$, define $\chi_n$ to be a smooth
function satisfying
\[
\chi_n=\left\{\begin{array}{cl}
1 &\;\;\;|x+x_n|>|x_n|/2,\\
0 &\;\;\;|x+x_n|<|x_n|/4,\;\; 
\end{array}\right.
\]
with $\sup_{x}|\partial^\alpha\chi_n(x)|\lesssim |x_n|^{-|\alpha|}$ for all multi-indices $\alpha$. Note that $\chi_n(x) \rightarrow 1$ as $n \rightarrow \infty$ for each $x \in \mathbb{R}^N$. 

For $T > 0$, define
\[
\tilde{v}_{n,T}(t, x) =\begin{cases}
 \chi_n(x-x_n)e^{it\Delta}P_n\psi(x -x_n) &\;\;|t|\leq T,\\
 e^{i(t-T)\Delta}\left[\tilde{v}_{n,T}(T)\right]&\;\;t>T,\\
 e^{i(t+T)\Delta}\left[\tilde{v}_{n,T}(-T)\right]&\;\;t<-T,
\end{cases}
\]
where we have set $P_n = P_{\leq|x_n|^\theta}$ for some small $0 < \theta <1$ to be determined more precisely below.  The rationale for this particular design of approximate solution is explained in \cite[Remark~4.1]{MMZ}. 

We will show the existence of the solutions $v_n$ by applying stability result (Proposition~\ref{stability}). To this end, we must  verify the following claims.

\ {\bf Claim $1$.} 
\[
\limsup\limits_{T\rightarrow \infty}\limsup\limits_{n\rightarrow \infty} \|\tilde{v}_{n,T}\|_{L^\infty_t H^1_x} + \|\tilde{v}_{n,T}\|_{S(\dot{H}^{s_c})}\lesssim 1.
\]
We first  consider the case $|t|\leq T$. First, by Sobolev embedding,
\begin{align*}
\|\tilde{v}_{n,T}\|_{L^\infty_{[|t|\leq T]}H^1_x}& \lesssim \|e^{it\Delta}P_n\psi\|_{L^\infty_tL^2_x}+ \|\nabla \chi_n\|_{L_x^N}\|e^{it\Delta} P_n \psi\|_{L^\infty_tL_x^{\frac{2N}{N-2}}}\\
& \quad + \| \chi_n\|_{L_x^\infty}\|\nabla e^{it\Delta} P_n \psi\|_{L^\infty_tL^2_x} \\
& \lesssim \|\psi\|_{H^1},
\end{align*}
for all $n$.  On the other hand, using Strichartz estimates, we obtain
\[
\|\tilde{v}_{n,T}\|_{S(\dot{H}^{s_c})}\lesssim \|P_n \psi\|_{H^{s_c}}\lesssim \|\psi\|_{H^1},
\]
for all $n$.  For the case $|t|>T$, we can again readily obtain $H^1$ boundedness, while (for $t\in[T,\infty)$, say) we use Strichartz to obtain
\[
\|\tilde{v}_{n,T}\|_{S(\dot{H}^{s_c})}\lesssim \|e^{-iT\Delta} \tilde{v}_{n,T}(T)\|_{H^{s_c}}\lesssim \|\tilde{v}_{n,T}(T)\|_{H^1}\lesssim\|\psi\|_{H^1},
\]
for all $n$.  Thus we conclude that Claim $1$ holds. 

{\bf Claim $2$.} 
\[
\lim\limits_{T\rightarrow \infty}\limsup\limits_{n\rightarrow \infty} \|\tilde{v}_{n,T}(t_n)-\psi_n\|_{ H^{1}}=0.
\]	
We consider two cases: First, if $t_n\equiv 0$, then we apply the dominated convergence theorem and the fact that  $P_n \psi \rightarrow \psi$ strongly in $H^1$ to obtain
\[
\|\tilde{v}_{n,T}(t_n)-\psi_n\|_{H^1}=\|\chi_n P_n\psi-\psi\|_{H^1}\leq \|\chi_n P_n\psi-P_n\psi\|_{H^1}+\| P_n\psi-\psi\|_{H^1}\;\rightarrow 0,
\]
as $n\to\infty$.  Next, if $t_n\to \infty$ (say), then we estimate 
\begin{align*}
\|\tilde{v}_{n,T}(t_n)-\psi_n\|_{H^1} &=\|e^{-iT\Delta}\chi_n e^{iT\Delta} P_n\psi-\psi\|_{H^1} \\
& = \|\chi_n e^{iT\Delta} P_n\psi-e^{iT\Delta}\psi\|_{H^1}\\
& \leq \|(1-\chi_n)e^{iT\Delta} P_n\psi\|_{H^1}+\|P_n\psi-\psi\|_{H^1},
\end{align*}
which tends to zero as $n\to \infty$.

{\bf Claim $3$.} 
\begin{equation}\label{embedding-errors}
\lim\limits_{T\rightarrow \infty}\limsup\limits_{n\rightarrow \infty} \left\{\|e_n\|_{S'(\dot{H}^{-s_c})}+\|e_n\|_{S'(L^2)}+\|\nabla e_n\|_{S'(L^2)}\right\}=0,
\end{equation}
where 
\[
e_{n}= (i\partial_t + \Delta)\tilde{v}_{n,T} + |x|^{-b}|\tilde{v}_{n,T}|^{\alpha}\tilde{v}_{n,T}.
\] 
Note that the errors depend on $T$ as well as $n$, although our notation does not indicate this explicitly. 

To prove this claim, we consider separately the regions $|t|\leq T$ and $|t|>T$.

{\bf Case $1$: $|t|\leq T$}. In this region, we have 
\[
\tilde{v}_{n,T}(t,x)=[\chi_n e^{it\Delta}(t)P_n \psi](x-x_n).
\]
Thus we may write
\[
e_{n}=|x|^{-b}|\tilde{v}_{n,T}|^\alpha \tilde{v}_{n,T}+\left(\Delta \chi_n w_n+2\nabla \chi_n \nabla w_n\right):=e_{n,1}+e_{n,2}, 
\]
where 
\[
w_n(t,x-x_n):=e^{it\Delta}P_n\psi(x-x_n).
\]

We first estimate $e_{n,2}$.  First,
\begin{align*}
\|e_{n,2}\|_{L^1_{[t\leq T]}L^2_x} & \lesssim  T \|\Delta \chi_n\|_{L_x^\infty}\|w_n\|_{L_t^\infty L^2_x}+T \|\nabla \chi_n\|_{L_x^\infty}\|\nabla w_n\|_{L^\infty_t L^2_x}\\
& \lesssim  T \left( |x_n|^{-2}\|\psi\|_{L_x^2}+|x_n|^{-1}\|\nabla \psi\|_{L_x^2} \right)\\
& \lesssim T(|x_n|^{-2}+|x_n|^{-1})\|\psi\|_{H^1}\,\rightarrow 0,\;\;\textnormal{as}\;\;n\rightarrow \infty.  
\end{align*}
On the other hand, using Bernstein's inequality \eqref{BerIn} as well,
\begin{align*}
\|\nabla e_{n,2}\|_{L^1_{[t\leq T]}L^2_x} & \lesssim   T \left( |x_n|^{-3}+|x_n|^{-2}\right)\|\psi\|_{H^1}+T|x_n|^{-1}\|\Delta w_n\|_{L^2}\\
&\lesssim T(|x_n|^{-3}+|x_n|^{-2}+|x_n|^{-1+\theta})\|\psi\|_{H^1}\,\rightarrow 0,\;\;\textnormal{as}\;\;n\rightarrow \infty.  
\end{align*}

For the $S'(\dot H^{-s_c})$-norm, let  $2<q<\infty$ and $2<r< \tfrac{2N}{N-2}$ be an $\dot H^{-s_c}$-admissible pair, so that
\[
\tfrac{2}{q} + \tfrac{N}{r} = \tfrac{N}{2}+s_c.
\]
Then we may estimate the $L_t^{q'}L_x^{r'}$-norm on $|t|\leq T$ as follows: 
\begin{align*}
\|e_{n,2}\|_{L_t^{q'}L_x^{r'}} & \lesssim T^{\frac{1}{q'}}\left( \|\Delta \chi_n\|_{L_x^{\frac{2r}{r-2}}} + \|\nabla \chi_n\|_{L_x^{\frac{2r}{r-2}}} \right) \| w_n\|_{L_t^\infty H^1} \\
& \lesssim T^{\frac{1}{q'}}\left( |x_n|^{-2+\frac{N}{2}-\frac{N}{r}} + |x_n|^{-1+\frac{N}{2}-\frac{N}{r}}\right)\|\psi\|_{H^1} \\
& = o(1), \qtq{as}n\to\infty,
\end{align*}
since the condition $r<\frac{2N}{N-2}$ guarantees that the powers of $|x_n|$ appearing above are negative. 

It remains to estimate $e_{n,1}$.  Using the definition of $\chi_n$, we see that we have the estimate $|x|^{-b}\lesssim |x_n|^{-b}$ on the support of this term.  

We first fix an $L^2$-admissible pair $(q,r)$ with $2<q<\infty$ and $(1-\alpha)q<2$ (i.e. $q<\tfrac{2}{1-\alpha}$ if $\alpha<1$, with no further constraint when $\alpha\geq 1$). In particular, this guarantees that we have $2<q'(\alpha+1)<\infty$. We now estimate as follows (still on the space-time slab $\{|t|\leq T\}\times\R^N$):
\begin{align*}
\| e_{n,1}\|_{L_t^{q'}L_x^{r'}} & \lesssim \| |x|^{-b}\|_{L_x^\gamma(|x|>|x_n|)} \|w_n\|_{L_t^{q'(\alpha+1)}L_x^\rho}^{\alpha+1},
\end{align*}
where $\gamma$ will be chosen to be greater than $\tfrac{N}{b}$, and $\rho$ is then determined by the scaling relation
\[
\tfrac{\alpha+1}{\rho}=1-\tfrac{1}{r}-\tfrac{1}{\gamma},\qtq{i.e.}\rho=\tfrac{(\alpha+1)\gamma r}{\gamma (r-1)-r}. 
\]
Note in particular that the denominator appearing in the definition of $\rho$ is positive.  Indeed, this follows from the fact that $r>2>\tfrac{N}{N-b}$, which in turn uses $b<\tfrac{N}{2}$. We now verify that 
\[
\rho\geq r_0:=\tfrac{2Nq(\alpha+1)}{q(N(\alpha+1)-4)+4},
\] 
where $r_0$ is the exponent such that $(q'(\alpha+1),r_0)$ yields an $L^2$-admissible pair.  For this, we note that when $\gamma=\tfrac{N}{b}$, using that $(q,r)$is $L^2$-admissible the strict inequality $\rho>r_0$ reduces exactly to $\alpha>\tfrac{4-2b}{N}$. Thus, by choosing $\gamma-\tfrac{N}{b}$ small enough (depending on $\alpha$), we may guarantee the desired condition for $\rho$.  We now continue from above, using Strichartz and Bernstein inequalities to obtain
\begin{align*}
\|e_{n,1}\|_{L_t^{q'}L_x^{r'}} & \lesssim |x_n|^{-b+\frac{N}{\gamma}}\||\nabla|^{\frac{1}{\alpha+1}\left(\frac{N\alpha}{2}+\frac{N}{\gamma}-2\right)} w_n\|_{L_t^{q'(\alpha+1)}L_x^{r_0}}^{\alpha+1} \\
& \lesssim |x_n|^{-b+\frac{N}{\gamma}+\frac{\theta}{\alpha+1}\left(\frac{N\alpha}{2}+\frac{N}{\gamma}-2\right)}\|\psi\|_{L_x^2}^{\alpha+1}. 
\end{align*}
Choosing $\theta$ small enough, we see that this quantity is $o(1)$ as $n\to\infty$.

We can estimate the derivative in the same space.  Indeed, if the derivative lands on $|x|^{-b}$ or on the cutoff $\chi_n$, we obtain the same estimate with an additional power of $|x_n|^{-1}$. If the derivative lands on $|w_n|^{\alpha}w_n$, then we proceed as above, obtaining a bound that involves the term $\|\psi\|_{L_x^2}^\alpha\|\nabla \psi\|_{L_x^2}$.  This is also acceptable.  

It remains to find a suitable $S'(\dot H^{-s_c})$ space in which to estimate. We  proceed similarly to the above, this time choosing $(\hat q,\hat r)$ to be an $\dot H^{-s_c}$-admissible pair with $2<\hat q<\infty$ and $(1-\alpha)\hat q<2$.  Note that it is possible to find an admissible pair with this choice of time exponent.  Indeed, combined with $\hat r<\tfrac{2N}{N-2}$ the constraints reduce to the inequality 
\[
\alpha + \tfrac{N}{2}-\tfrac{2-b}{\alpha}>0,
\]
which holds due to the assumption $\alpha>\tfrac{4-2b}{N}$. We begin with the same estimate as before:
\[
\|e_{n,1}\|_{L_t^{\hat q'}L_x^{\hat r'}} \lesssim \||x|^{-b}\|_{L^{\hat \gamma}(|x|>|x_n|)}\|w_n\|_{L_t^{\hat q'(\alpha+1)}L_x^{\hat\rho}}^{\alpha+1},
\]
with $\hat\gamma>\tfrac{N}{b}$ and 
\[
\tfrac{\alpha+1}{\hat\rho}=1-\tfrac{1}{\hat r}-\tfrac{1}{\hat\gamma},\qtq{i.e.}\hat\rho=\tfrac{(\alpha+1)\hat\gamma\hat r}{\hat\gamma(\hat r-1)-\hat r}. 
\]
As before, $\hat\gamma(\hat r-1)-\hat r>0$ follows from $\hat\gamma>\tfrac{N}{b}$, $b<\tfrac{N}{2}$ and $\hat r>2$. Similarly, we would like to guarantee that
\[
\hat \rho \geq \hat r_0 :=\tfrac{2N\hat q(\alpha+1)}{\hat q(N(\alpha+1)-4)+4} 
\]
where $\hat r_0$ is the exponent such that $(\hat q'(\alpha+1),\hat r_0)$ yields an $L^2$-admissible pair. For $\hat \gamma = \tfrac{N}{b}$, the strict inequality $\hat\rho > \hat r_0$ reduces to the inequality 
\[
p(\alpha):=N\alpha^2+(N-4+2b)\alpha-(4-2b)>0. 
\]
Now observe that $p(\tfrac{4-2b}{N})=0$ while $p'(\alpha)>0$ for $\alpha>\tfrac{4-2b}{N}$; indeed, the latter inequality reduces to $b<2+\tfrac{N}{2}$. In particular we have $p(\alpha)>0$ all cases under consideration, and the inequality is preserved by choosing $\hat\gamma - \tfrac{N}{b}>0$ sufficiently small. Continuing from above and estimating as before, we obtain
\begin{align*}
\|e_{n,1}\|_{L_t^{\hat q'}L_x^{r'}} & \lesssim |x_n|^{-b+\frac{N}{\hat \gamma}}\| |\nabla|^{\frac{1}{\alpha+1}\left(\frac{N\alpha}{2}+\frac{N}{\hat \gamma}-2+s_c\right)} w_n\|_{L_t^{\hat q'(\alpha+1)}L_x^{\hat r_0}}^{\alpha+1} \\
& \lesssim |x_n|^{-b+\frac{N}{\hat\gamma}+\frac{\theta}{\alpha+1}\left(\frac{N\alpha}{2}+\frac{N}{\hat \gamma}-2+s_c\right)}\|\psi\|_{L_x^2}^{\alpha+1}.
\end{align*}
In particular, this quantity is $o(1)$ as $n\to\infty$ provided $\theta$ is chosen sufficiently small.

{\bf Case $2$: $t>T$}.  In this regime, we have
\[
e_n = |x|^{-b}|\tilde v_{n,T}|^\alpha \tilde v_{n,T},\quad \tilde v_{n,T}(t,x)=e^{i(t-T)\Delta}[\chi_n(x-x_n)e^{iT\Delta}P_n\psi(x-x_n)]. 
\]
Here the vanishing in \eqref{embedding-errors} will essentially be a consequence of Strichartz estimates and the monotone convergence theorem.  We will rely on the nonlinear estimates appearing in Lemma~\ref{L:NL}. In particular, we may obtain
\begin{align*}
\|e_n&\|_{S'(L^2)}  + \|\nabla e_n\|_{S'(L^2)} + \|e_n\|_{S'(\dot H^{-s_c})} \\
& \lesssim \| \tilde v_{n,T}\|_{L_t^\infty H_x^1}^\theta\|\tilde v_{n,T}\|_{S(\dot H^{s_c})}^{\alpha-\theta}\left(\|\tilde v_{n,T}\|_{S(\dot H^{s_c})}+\|\tilde v_{n,T}\|_{S(L^2)}+\|\nabla \tilde v_{n,T}\|_{S(L^2)}\right) \\
& \quad + \|\tilde v_{n,T}\|_{L_t^\infty H_x^1}^{1+\theta}\|\tilde v_{n,T}\|_{S(\dot H^{s_c})}^{\alpha-\theta},
\end{align*}
for some $\theta\in(0,\alpha)$, where all spacetime norms are restricted to $[T,\infty)\times\R^N$. In practice, each norm above is represented by some particular admissible pair.

Thus it suffices to establish the following:
\begin{itemize}
\item[(a)] $\tilde v_{n,T}$ is uniformly bounded in $H^1$ for $t\geq T$.
\item[(b)] We have
\[
\lim_{T\to\infty}\lim_{n\to\infty}[\|\tilde v_{n,T}\|_{S(L^2)} + \|\nabla\tilde v_{n,T}\|_{S(L^2)}+\|\tilde v_{n,T}\|_{S(\dot H^{s_c})}]=0,
\]
where all space-time norms are restricted to $[T,\infty)\times\R^N$.
\end{itemize}
Item (a) is clear, as $P_n$, multiplication by $\chi_n$, and $e^{i\cdot\Delta}$ are all uniformly bounded operators from $H^1\to H^1$.  Therefore we turn to item (b). We first consider the term $\nabla \tilde v_{n,T}$ in $S(L^2)$. We begin by fixing a representative $L^2$-admissible pair $(q,r)$ and estimating
\begin{align*}
\|\nabla \tilde v_{n,T}\|_{L_t^q L_x^r(\{t>T\})} & \lesssim \|\nabla e^{it\Delta}[\chi_n e^{iT\Delta}P_n\psi]\|_{L_t^qL_x^r(\{t>0\})} \\
& \lesssim \| \nabla e^{it\Delta}P_n\psi\|_{L_t^q L_x^r(\{t>T\})} + \|\nabla\left[(\chi_n-1)e^{iT\Delta}P_n\psi\right]\|_{L_x^2}. 
\end{align*}
For fixed $T$, the second term above tends to zero as $n\to\infty$ by the dominated convergence theorem.  On the other hand, by Strichartz, the first term is bounded by $\|\psi\|_{L_x^2}$ for all $T\geq 0$, and hence tends to zero by the monotone convergence theorem. 

Similar arguments suffice to control the $S(L^2)$ and $S(\dot H^{s_c})$ norms, and hence we conclude that item (b) holds. This completes the treatment of Case 2, and hence we finish the proof of \eqref{embedding-errors} and Claim 3. 

With the three claims in place, we may apply the stability result (Proposition~\ref{stability}) to deduce the existence of a global solution $v_n$ to \eqref{INLS} satisfying
$v_n(0) = \psi_n$ and 
\[
\|v_n\|_{S(\dot{H}^{s_c})}
+ \|v_n\|_{S(L^2)} + \|\nabla v_n\|_{S(L^2)} \lesssim 1,
\]
for all $n$ sufficiently large. The proof of \eqref{aprox} now follows as in \cite[Proposition~$3.3$]{MMZ}.  This completes the proof. \end{proof}

\section{Proof of the main result}\label{S:proof}

The main result, Theorem~\ref{T} is proven by contradiction.  Supposing that the theorem fails for some choice of $(N,b,\alpha)$, Proposition~\ref{P:exist} guarantees the existence of a minimal blowup solution, namely, a non-zero global solution $u$ to \eqref{INLS} that obeys (i) the orbit $\{u(t):t\in\R\}$ is pre-compact in $H^1$ and (ii) the solution lives below the ground state threshold, i.e. 
\[
M[u]^{1-s_c}E[u]^{s_c}<M[Q]^{1-s_c}E[Q]^{s_c}\,\,\,\textnormal{and}\,\,\,\|u(t)\|_{L^2}^{1-s_c}\|\nabla u(t)\|_{L^2}^{s_c} < \|Q\|_{L^2}^{1-s_c} \|\nabla Q\|_{L^2}^{s_c}. 
\]

At this point, we may now argue exactly as in \cite{FG2} and utilize a localized virial argument to derive a contradiction, thus establishing Theorem~\ref{T}.  In particular, the necessary coercivity in the virial argument comes from Lemma~\ref{L:coercive}(iii), while the pre-compactness in $H^1$ allows for the virial identity to be localized in a suitable manner. 

In particular, we let $R>0$ and introduce the quantity $z_R$ as in Section~\ref{S:virial}. Using uniform $H^1$ boundedness, Lemma~\ref{L:coercive}(iii), \eqref{zRbd}, and \eqref{errors}, we then utilize the fundamental theorem of calculus (in the form $\int_0^T z_R''(t)\,dt = z_R'(T)-z_R'(0)$) to obtain the following estimate for an arbitrary $T>0$:
\begin{align*}
\delta&\int_0^T \|\nabla u(t)\|_{L^2}^2 \,dt \\
& \lesssim R\|u\|_{L_t^\infty H_x^1}^2 + \int_0^T \int_{|x|>R} |\nabla u(t,x)|^2 + \tfrac{1}{R^2}|u(t,x)|^2 + \tfrac{1}{R^b}|u(t,x)|^{\alpha+2}\,dx\,dt
\end{align*}
for some $\delta>0$.  In particular, using pre-compactness in $H^1$, we let $\eps>0$ to be determined below and choose $R=R(\eps)>0$ sufficiently large (independent of $T$) to obtain
\[
\delta\int_0^T \|\nabla u(t)\|_{L^2}^2\,dx \lesssim R\|u\|_{L_t^\infty H_x^1}^2+\eps T.
\]
As Lemma~\ref{L:coercive}(i) (and the fact that $u\not\equiv 0$) implies the uniform lower bound $\|\nabla u(t)\|_{L^2}^2\gtrsim E[u_0]>0$, we may now deduce that
\[
E[u_0] \lesssim \tfrac{R}{T\delta} + \tfrac{\eps}{\delta}
\]
for any $T>0$.  In particular, choosing $\eps=\eps(\delta)$ sufficiently small and $T=T(R(\eps),\delta)$ sufficiently large, we can now obtain a contradiction to the fact that $E[u_0]>0$. This precludes the possibility of minimal blowup solutions as in Proposition~\ref{P:exist} and hence completes the proof of Theorem~\ref{T}.




\end{document}